\documentclass{article}

\usepackage{amsthm, amssymb, amsmath}
\usepackage{verbatim}
\usepackage{graphics}      
\usepackage{graphicx}      
\usepackage{setspace}
\usepackage{url}

\newtheorem{thm}{Theorem}
\newtheorem{prop}[thm]{Proposition}
\newtheorem{cor}[thm]{Corollary}
\newtheorem{lem}[thm]{Lemma}

\theoremstyle{plain}

\newcommand{\Z}{\mathbb{Z}}

\newcommand{\Q}{\mathbb{Q}}
\newcommand{\C}{\mathbb{C}}

\newcommand{\dotp}[2]{\left\langle #1,#2 \right\rangle}

\title{On the Ranks of the 2-Selmer Groups of Twists of a Given Elliptic Curve}
\author{Daniel M. Kane \\
Stanford University Department of Mathematics\\
Building 380, Sloan Hall\\
Stanford, CA 94305\\
\texttt{dankane@math.stanford.edu}
}

\begin{document}

\maketitle
\begin{abstract}
In \cite{kn:SD}, Swinnerton-Dyer considered the proportion of twists of an elliptic curve with full 2-torsion that have 2-Selmer group of a particular dimension. Swinnerton-Dyer obtained asymptotic results on the number of such twists using an unusual notion of asymptotic density. We build on this work to obtain similar results on the density of twists with particular rank of 2-Selmer group using the natural notion of density.
\end{abstract}
\section{Introduction}

Let $c_1,c_2,c_3$ be distinct rational numbers.  Let $E$ be the elliptic curve defined by the equation
$$
y^2 = (x-c_1)(x-c_2)(x-c_3).
$$
We make the additional technical assumption that none of the $(c_i-c_j)(c_i-c_k)$ are squares.  This is equivalent to saying that $E$ is an elliptic curve over $\Q$ with complete 2-torsion and no cyclic subgroup of order 4 defined over $\Q$.  For $b$ a square-free number, let $E_b$ be the twist defined by the equation
$$
y^2 = (x-bc_1)(x-bc_2)(x-bc_3).
$$
Let $S$ be a finite set of places of $\Q$ including $2,\infty$ and all of the places at which $E$ has bad reduction.  Let $D$ be a positive integer divisible by 8 and by the primes in $S$.  Let $S_2(E_b)$ denote the 2-Selmer group of the curve $E_b$.  We will be interested in how the rank varies with $b$ and in particular in the asymptotic density of $b$'s so that $S_2(E_b)$ has a given rank.

The parity of $\dim(S_2(E_b))$ depends only on the class of $b$ as an element of $\prod_{\nu\in S} \Q_\nu^*/(\Q_\nu^*)^2$.  We claim that for exactly half of these values this dimension is odd and exactly half of the time it is even. In particular, we make the following claim which will be proved later.
\begin{lem}\label{parityLem}
There exists a set $S$ consisting of exactly half of the classes $c$ in $(\Z/D)^*/((\Z/D)^*)^2$, so that for any positive integer $b$ relatively prime to $D$ we have that $\dim(S_2(E_b))$ is even if and only if $b$ represents a class in $S$.
\end{lem}
We put off the proof of this statement until Section \ref{MomentSec}.

Let $b=p_1p_2\ldots p_n$ where $p_i$ are distinct primes relatively prime to $D$.  In \cite{kn:SD} the rank of $S_2(E_b)$ is shown to depend only on the images of the $p_i$ in $(\Z/D)^*/((\Z/D)^*)^2$ and upon which $p_i$ are quadratic residues modulo which $p_j$.  There are $2^{n|S|+\binom{n}{2}}$ possible sets of values for these.  Let $\pi_d(n)$ be the fraction of this set of possibilities that cause $S_2(E_b)$ to have rank exactly $d$.  Then the main theorem of \cite{kn:SD} together with Lemma \ref{parityLem} implies that:
\begin{thm}\label{SDThm}
$$
\lim_{n\rightarrow\infty} \pi_d(n) = \alpha_d.
$$
where
$\alpha_0=\alpha_1=0$ and $\alpha_{n+2} = \frac{2^n}{\prod_{j=1}^n(2^j-1)\prod_{j=0}^\infty(1+2^{-j})}.$
\end{thm}
The actual theorem proved in \cite{kn:SD} says that if, in addition, the class of $b$ in $\prod_{\nu\in S} \Q_\nu^*/(\Q_\nu^*)^2$ is fixed, then the analogous $\pi_d(n)$ either converge to $2\alpha_d$ for $d$ even and $0$ for $d$ odd, or to $2\alpha_d$ for $d$ odd and $0$ for $d$ even.

This tells us information about the asymptotic density of twists of $E$ whose 2-Selmer group has a particular rank.  Unfortunately, this asymptotic density is taken in a somewhat awkward way by letting the number of primes dividing $b$ go to infinity.  In this paper, we prove the following more natural version of Theorem \ref{SDThm}:
\begin{thm}\label{mainThmSel}
Let $E$ be an elliptic curve over $\Q$ with full 2-torsion defined over $\Q$ so that in addition for $E$ we have that
$$
\lim_{n\rightarrow\infty} \pi_d(n) = \alpha_d.
$$
With $\alpha_d$ as given in Theorem \ref{SDThm}.  Then
$$
\lim_{N\rightarrow \infty} \frac{\#\{b\leq N : b \ \textrm{square-free}, \ (b,D)=1 \ \textrm{and} \ \dim(S_2(E_b))=d\}}{\#\{b\leq N : b \ \textrm{square-free and} \ (b,D)=1\}} = \alpha_d.
$$
\end{thm}
Applying this to twists of $E$ by divisors of $D$ and noting that twists by squares do not affect the Selmer rank we have that
\begin{cor}
$$
\lim_{N\rightarrow\infty} \frac{\#\{b\leq N: \dim(S_2(E_b))=d\}}{N} = \alpha_d.
$$
\end{cor}
and
\begin{cor}
$$
\lim_{N\rightarrow\infty} \frac{\#\{-N\leq b\leq N: \dim(S_2(E_b))=d\}}{2N} = \alpha_d.
$$
\end{cor}

Our technique is fairly straightforward.  Our goal will be to prove that the average moments of the size of the Selmer groups will be as expected.  As it turns out, this along with Lemma \ref{parityLem} will be enough to determine the probability of seeing a given rank.  In order to analyze the Selmer groups we follow the method described in \cite{kn:SD}.  Here the 2-Selmer group of $E_b$ can be expressed as the intersection of two Lagrangian subspaces, $U$ and $W$, of a particular symplectic space, $V$, over $\mathbb{F}_2$.  Although $U,V$ and $W$ all depend on $b$, once the number of primes dividing $b$ has been fixed along with its congruence class modulo $D$, these spaces can all be written conveniently in terms of the primes, $p_i$, dividing $b$, which we think of as formal variables.  Using the formula $|U\cap W| = \frac{1}{\sqrt{|V|}}\sum_{u\in U,w\in W} (-1)^{u\cdot w}$, we reduce our problem to bounding the size of the ``characters'' $(-1)^{u\cdot w}$ when averaged over $b$.  These ``characters'' turn out to be products of Dirichlet characters of the $p_i$ and Legendre symbols of pairs of the $p_i$.  The bulk of our analytic work is in proving these bounds.  These bounds will allow us to discount the contribution from most of the terms in our sum (in particular the ones in which Legendre symbols show up in a non-trivial way), and allow us to show that the average of the remaining terms is roughly what should be expected from Swinnerton-Dyer's result.

We should point out the connections between our work and that of Heath-Brown in \cite{kn:HB} where he proves our main result for the particular curve
$$
y^2 = x^3-x.
$$
We employ techniques similar to those of \cite{kn:HB}, but the algebra behind them is organized significantly differently.  Heath-Brown's overall strategy is again to compute the average sizes of moments of $|S_2(E_b)|$ and use these to get at the ranks.  He computes $|S_2(E_b)|$ using a different formula than ours.  Essentially what he does is use some tricks specific to his curve to deal with the conditions relating to primes dividing $D$, and instead of considering each prime individually, he groups them based on how they occur in $u$ and $w$.  He lets $D_i$ be the product of all primes dividing $b$ that relate in a particular way (indexed by $i$).  He then gets a formula for $|S_2(E_b)|$ that's a sum over ways of writing $b$ as a product, $b=\prod D_i$, of some term again involving characters of the $D_i$ and Legendre symbols.  Using techniques similar to ours he shows that terms in this sum where the Legendre symbols have a non-negligible contribution (are not all trivial due to one of the $D_i$ being 1) can be ignored.  He then uses some algebra to show that the average of the remaining terms is the desired value.  This step differs from our technique where we merely make use of Swinnerton-Dyer's result to compute our average.  Essentially we show that the algebra and the analysis for this problem can be done separately and use \cite{kn:SD} to take care of the algebra.  Finally, Heath-Brown uses some techniques from linear algebra to show that the moment bounds imply the correct densities of ranks, while we use techniques from complex analysis.

We also note the work of Yu in \cite{Yu}. In this paper, Yu shows that for a wide family of curves of full $2$-torsion that the average size of the 2-Selmer group of a twist is equal to 12. This work uses techniques along the lines of Heath-Brown's, though has some added complication in order to deal with the greater generality.

One advantage of our technique over these others is that we can, to some degree, separate the algebra involved in analyzing the sizes of these Selmer groups from the analysis. When considering the distribution of ranks of Selmer groups of twists of an elliptic curve, there are two types of density estimates that have come up in the literature. The first is to use the natural notion of density over some obvious ordering of twist parameter. The other is to use some notion similar to that of Swinnerton-Dyer, which can be thought of as letting the number of primes dividing the twist parameter go to infinity. Although one is usually interested in natural densities, the Swinnerton-Dyer type results are often easier to prove as they tend to be essentially algebraic in nature, while results about natural density will generally require some tricky analytic work. The techniques of this paper show how asymptotics of the Swinnerton-Dyer type can be upgraded to results for natural density. Although we have only managed to carry out this procedure for the family of curves used in Theorem \ref{SDThm}, there is hope that this procedure might have greater applicability. For example, if someone were to obtain a Swinnerton-Dyer type result for twists of an elliptic curve with full 2-torsion over $\Q$ that \emph{has} a rational 4-isogeny, it is almost certain that the techniques from this paper would allow one to obtain a result for the same curve using the natural density. Additionally, in \cite{nottor}, Klagsbrun, Mazur and Rubin consider the ranks of twists of an elliptic curve with $\textrm{Gal}(K(E[2])/K)\simeq S_3$, and obtain Swinnerton-Dyer type density results. It is possible that ideas in this paper may be adapted to improve these results to work with a more natural notion of density as well. Unfortunately, working in this extended context will likely complicate the analytic aspects of the argument considerably. For example, while we make important use of the fact that the rank of $S_2(E_b)$ depends only on congruence classes of primes dividing $b$ and Legendre symbols between them, it is shown in \cite{spin} that for curves with cyclic cubic field of $2$-torsion that the Selmer rank can depend on more complicated algebraic objects (such as what they term the spin of a prime).

In Section \ref{prelimSecSel}, we introduce some basic concepts that will be used throughout.  In Section \ref{charBoundSec}, we will prove the necessary character bounds.  We use these bounds in Section \ref{MomentSec} to establish the average moments of the size of the Selmer groups.  Finally, in Section \ref{toRankSec}, we explain how these results can be used to prove our main Theorem.

\section{Preliminaries}\label{prelimSecSel}

\subsection{Asymptotic Notation}

Throughout the rest of this paper we will make extensive use of $O$, and similar asymptotic notation.  In our notation, $O(X)$ will denote a quantity that is at most $H\cdot X$ for some \emph{absolute} constant $H$.  If we need asymptotic notation that depends on some parameters, we will use $O_{a,b,c}(X)$ to denote a quantity that is at most $H(a,b,c)\cdot X$, where $H$ is some function depending only on $a,b$ and $c$.

\subsection{Number of Prime Divisors}

In order to make use of Swinnerton-Dyer's result, we will need to consider twists of $E$ by integers $b\leq N$ with a specific number of prime divisors.  For an integer $m$, we let $\omega(m)$ be the number of prime divisors of $m$.  In our analysis, we will need to have estimates on the number of of such $b$ with a particular number of prime divisors.  We define
$$
\Pi_n(N) = \#\{\textrm{primes } p\leq N \textrm{ so that } \omega(p)=n\}.
$$
In order to deal with this we use Lemma A of \cite{RamPrimes} which states:
\begin{lem}\label{primeDivsLem}
There exist absolute constants $C$ and $K$ so that for any $\nu$ and $x$
$$
\Pi_{\nu+1}(x) \leq \frac{Kx}{\log(x)} \frac{(\log\log x + C)^\nu}{\nu!}.
$$
\end{lem}

By maximizing the above in terms of $\nu$ it is easy to see that
\begin{cor}\label{maxDivDensityCor}
$$
\Pi_n(N) = O\left(\frac{N}{\sqrt{\log\log(N)}}\right).
$$
\end{cor}

It is also easy to see from the above that most integers of size roughly $N$ have about $\log\log(N)$ prime factors.  In particular:
\begin{cor}\label{loglogDivsCor}
There is a constant $c>0$ so that for all $N$, the number of $b\leq N$ with $|\omega(b)-\log\log(N)|> \log\log(N)^{3/4}$ is at most
$$
2N \exp\left( -c \sqrt{\log\log(N)}\right).
$$
In particular, the fraction of $b\leq N$ with $|\omega(b)-\log\log(N)| < \log\log(N)^{3/4}$ goes to 1 as $N$ goes to infinity.
\end{cor}
We will use Corollary \ref{loglogDivsCor} to restrict our attention only to twists by $b$ with an appropriate number of prime divisors.

\section{Character Bounds}\label{charBoundSec}

Our main purpose in this section will be to prove the following Propositions:
\begin{prop}\label{characterBoundProp}
Fix positive integers $D,n,N$ with $4|D$, $\log\log N>1$, and $(\log\log N)/2 < n < 2\log\log N$, and let $c>0$ be a real number.  Let $d_{i,j},e_{i,j}\in \Z/2$ for $i,j=1,\ldots,n$ with $e_{i,j}=e_{j,i},d_{i,j}=d_{j,i},e_{i,i}=d_{i,i}=0$ for all $i,j$.  Let $\chi_i$ be a quadratic character with modulus dividing $D$ for $i=1,\ldots,n$.  Let $m$ be the number of indices $i$ so that at least one of the following hold:
\begin{itemize}
\item $e_{i,j}=1$ for some $j$ or
\item $\chi_i$ has modulus not dividing 4 or
\item $\chi_i$ has modulus exactly 4 and $d_{i,j}=0$ for all $j$.
\end{itemize}
Let $\epsilon(p)=(p-1)/2.$ Then if $m>0$
\begin{align}\label{charBoundEquation}
\left|\frac{1}{n!}
\sum_{S_{N,n,D}}
\prod_i \chi_i(p_i)
\prod_{i<j} (-1)^{\epsilon(p_i)\epsilon(p_j)d_{i,j}}\prod_{i<j} \left(\frac{p_i}{p_j}\right)^{e_{i,j}}
\right|
= O_{c,D}\left(Nc^{m}\right),
\end{align}
where $S_{N,n,D}$ is the set of $n$-tuples of distinct primes $p_1,\ldots,p_n$ so that $b=p_1\cdots p_n$ is relatively prime to $D$ and of size at most $N$.
\end{prop}

Note that $m$ is the number of indices $i$ so that no matter how we fix the values of $p_j$ for the $j\neq i$ that the summand on the left hand side of Equation \ref{charBoundEquation} still depends on $p_i$. The index set $S_{N,n,D}$ above is a way of indexing (up to overcounting by a factor of $n!$) the set of integers $b\leq N$ that are squarefree, relatively prime to $D$ and have $\omega(b)=n$. This notation will be used throughout the rest of the paper. The sum in equation \eqref{charBoundEquation} can be thought of as a sum over such $b$ (the $1/n!$ term accounts for the overcounting) of a ``character'' defined by the $\chi_i,d_{i,j}$ and $e_{i,j}$. Proposition \ref{characterBoundProp} will allow us to show that the ``characters'' in which the Legendre symbols make a non-trivial appearance add a negligible contribution to our moments.

\begin{prop}\label{averageClassProp}
Let $n,N,D$ be positive integers with $\log\log N > 1$, and $(\log\log N)/2 < n < 2\log\log N$.  Let $G=((\Z/D)^*/((\Z/D)^*)^2)^n$.  Let $f:G\rightarrow \C$ be a function with $|f|_\infty\leq 1$.  Then
\begin{align}\label{averageClassEquation}
\frac{1}{n!}&\sum_{S_{N,n,D}} f(p_1,\ldots,p_n)= \left( \frac{1}{|G|}\sum_{g\in G}f(g)\right)\left(\frac{|S_{N,n,D}|}{n!}\right) + O_D\left( \frac{N(\log\log\log N)}{\log\log N}\right).
\end{align}
(Above $f(p_1,\ldots,p_n)$ is really $f$ applied to the vector of their reductions modulo $D$).
\end{prop}

This Proposition says that the average of $f$ over such $S_{N,n,D}$ is roughly equal to the average of $f$ over $G$.  This will allow us to show that the average value of the remaining terms in our moment calculation equal what we would expect given Swinnerton-Dyer's result.

We begin with a Proposition that gives a more precise form of Proposition \ref{characterBoundProp} in the case when the $e_{i,j}$ are all 0.
\begin{prop}\label{simpleCharBoundProp}
Let $D,n,N$ be integers with $4|D$ with $\log\log N>1$.  Let $C>0$ be a real number.  Let $d_{i,j}\in \Z/2$ for $i,j=1,\ldots,n$ with $d_{i,j}=d_{j,i},d_{i,i}=0$.  Let $\chi_i$ be a quadratic character of modulus dividing $D$ for $i=1,\ldots,n$.  Suppose that no Dirichlet character of modulus dividing $D$ has an associated Siegel zero larger than $1-\beta^{-1}$.  Let $$B=\max(e^{(C+2)\beta \log\log N},e^{K(C+2)^2(\log D)^2 (\log\log (DN))^2},n\log^{C+2}(N))$$ for $K$ a sufficiently large absolute constant.  Suppose that $B^n < \sqrt{N}$.  Let $m$ be the number of indices $i$ so that either:
\begin{itemize}
\item $\chi_i$ does not have modulus dividing 4 or
\item $\chi_i$ has modulus exactly 4 and $d_{i,j}=0$ for all $j$.
\end{itemize}
Then
\begin{align}\label{simpleCharacterBoundEquation}
&\left|\frac{1}{n!}
\sum_{S_{N,n,D}}
\prod_i \chi_i(p_i)
\prod_{i<j} (-1)^{\epsilon(p_i)\epsilon(p_j)d_{i,j}}
\right|\\
& \ \ \ \ \ = O\left(\frac{N}{\sqrt{\log\log(N)}}\right)\left(O\left(\frac{\log\log B}{n}\right)^{m}+ (\log N)^{-C} \right).\notag
\end{align}
\end{prop}
Note once again that $m$ is the number of $i$ so that if the values of $p_j$ for $j\neq i$ are all fixed, the resulting summand will still depend on $p_i$.

The basic idea of the proof will be by induction on $m$.  If $m=0$, we can bound by the number of terms in our sum, giving a bound of $\Pi_n(N)$, which we bound using Corollary \ref{maxDivDensityCor}.  If $m>0$, there is some $p_i$ so that no matter how we set the other $p_j$, our character still depends on $p_i$.  We split into cases based on whether $p_i>B$.  If $p_i>B$, we fix the values of the other $p_j$, and use bounds on character sums.  For $p_i\leq B$, we note that this happens for only about a $\frac{\log\log B}{n}$ fraction of the terms in our sum, and for each possible value of $p_i$ inductively bound the remaining sum. To deal with the first case we prove the following Lemma:

\begin{lem}\label{CharacterSumLem}
Let $K$ be a sufficiently large constant.  Take $\chi$ any non-trivial Dirichlet character of modulus at most $D$ and with no Siegel zero more than $1-\beta^{-1}$, $N,C>0$ integers, and $X$ any integer with
$$
X> \max(e^{(C+2)\beta \log\log N},e^{K(C+2)^2(\log D)^2 (\log\log (DN))^2}).
$$
Then
$$
\left|\sum_{p\leq X} \chi(p)\right| \leq O( X \log^{-C-2}(N)).
$$
Where the sum above is over primes $p$ less than or equal to $X$.
\end{lem}
\begin{proof}
Theorem 5.27 of \cite{kn:IK} implies that for any $Y$ that for some constant $c>0$,
$$
\sum_{n\leq Y} \chi(n)\Lambda(n) = Y\cdot O\left(Y^{-\beta^{-1}} + \exp\left(\frac{-c\sqrt{\log(Y)}}{\log D} \right)(\log D)^4 \right).
$$
Note that the contribution to the above coming from $n$ a power of a prime is $O(\sqrt{Y})$.  Using Abel summation to reduce this to a sum over $p$ of $\chi(p)$ rather than $\chi(p)\log(p)$, we find that
$$
\sum_{p\leq X} \chi(p) \leq X\cdot O\left(X^{-\beta^{-1}} + \exp\left(\frac{-c\sqrt{\log(X)}}{\log D} \right)(\log D)^4 \right)+O(\sqrt{X}).
$$
The former term is sufficiently small since by assumption $X>e^{(C+2)\beta \log\log N}$.  The latter term is small enough since $X>e^{K(C+2)^2(\log D)^2 (\log\log (DN))^2}.$  The last term is small enough since clearly $X>\log^{2C+4}(N)$.
\end{proof}

For positive integers $n,N,D$, and $S$ a set of prime numbers, denote by $Q(n,N,D,k,S)$ the maximum possible absolute value of a sum of the form given in Equation \eqref{simpleCharacterBoundEquation} with $m\geq k$, with the added restriction that none of the $p_i$ lie in $S$.  In particular a sum of the form
$$
\frac{1}{n!}\sum_{S_{N,n,D'}}
\prod_i \chi_i(p_i)
\prod_{i<j} (-1)^{\epsilon(p_i)\epsilon(p_j)d_{i,j}}
$$
where $\chi_i$ are characters of modulus dividing $D$, $d_{i,j}\in\{0,1\}$, and $$D'=D\cdot \prod_{p\in S}p.$$

We write the inductive step for our main bound as follows:
\begin{lem}\label{simpleInductiveStepLem}
For integers, $n,D,N,M,C$ and $B$ with $$B>\max(e^{(C+2)\beta \log\log M},e^{K(C+2)^2(\log D)^2 (\log\log (DM))^2},n\log^{C+2}(M)),$$ where $1-\beta^{-1}$ is the largest Siegel zero of a Dirichlet character of modulus dividing $D$ and $K$ a sufficiently large constant, $1\leq k\leq n$ and $S$ a set of primes $\leq B$, then $Q(n,N,D,k,S)$ as described above is at most
\begin{align*}
O( N \log(N)\log^{-C-2}(M)) + \frac{1}{n}\sum_{\begin{subarray}{l} p < B \\ p \not\in S \end{subarray}} Q(n-1,N/p,D,k-1,S\cup\{p\}).
\end{align*}
\end{lem}
\begin{proof}
Since $k\geq 1$, there must be an $i$ so that either $\chi_i$ has modulus bigger than 4 or has modulus exactly 4 and all of the $d_{i,j}$ are 0.  Without loss of generality, $n$ is such an index.  We split our sum into cases depending on whether $p_n\geq B$.  For $p_n\geq B$, we proceed by fixing all of the $p_j$ for $j\neq n$ and summing over $p_n$.  Letting $P=\prod_{i=1}^{n-1} p_i$, we have
$$
\sum_{P=1}^{N/B} \frac{1}{n!}\sum_{\begin{subarray}{l} P = p_1\ldots p_{n-1}\\ p_i \textrm{ distinct} \\ p_i \not\in S \\ (D,P)=1 \end{subarray}} a \sum_{\begin{subarray}{l} B\leq p_n \leq N/P \\ p_n \neq p_j \end{subarray}} \chi(p_n)
$$
where $a$ is some constant of norm 1 depending on $p_1\ldots p_{n-1}$, and $\chi$ is a non-trivial character of modulus dividing $D$, perhaps also depending on $p_1,\ldots, p_{n-1}$.  The condition that $p_n\neq p_j$ alters the value of the inner sum by at most $n$.  With this condition removed, we may bound the inner sum by applying Lemma \ref{CharacterSumLem} (taking the difference of the terms with $X=N/P$ and $X=B$).  Hence the value of the inner sum is at most $O(N/P\log^{-C-2}(M)+n)$.  Since $N/P \geq B \geq n \log^{C+2}(M)$, this is just $O(N/P\log^{-C-2}(M))$. Note that for each $P$, there are at most $(n-1)!$ ways of writing it as a product of $n-1$ primes (since the primes will be unique up to ordering).  Hence, ignoring the extra $1/n$ factor, the sum above is at most
$$
\sum_{P=1}^{N/B} O(N/P \log^{-C-2}(M)) = O(N \log (N) \log^{-C-2}(M)).
$$

For $p_n< B$, we fix $p_n$ and consider the sum over the remaining $p_i$.  We note that for $p$ a prime not in $S$ and relatively prime to $D$, this sum is plus or minus one over $n$ times a sum of the type bounded by $Q(n-1,N/p,D,k-1,S\cup\{p\})$.  In particular, we note that since by assumption the value of $m$ for our original sum was at least $k$, that upon fixing this value of $p_n$, the value of $m$ for the resulting sum is at least $k-1$ and is thus bounded by $Q(n-1,N/p,D,k-1,S\cup\{p\})$.  This completes our proof.
\end{proof}

We are now prepared to Prove Proposition \ref{simpleCharBoundProp}
\begin{proof}
We prove by induction on $k$ that for $n,N,D,C,M,\beta,B$ as above with
$$B>\max(e^{(C+2)\beta \log\log M},e^{K(C+2)^2(\log D)^2 (\log\log (DM))^2},n\log^{C+2}(M)),$$ and $S$ a set of primes $\leq B$, and $c$ a sufficiently large constant that
\begin{align}\label{SimpleInductionEquation}
Q(n,N,D,k,S) \leq & c\left(\frac{N}{\sqrt{\log\log(N/B^n)}}\right)\left(\frac{c\log\log B}{n} \right)^k \\
 & + c N\log(N)\log^{-C-2}(M)\sum_{a=0}^{k-1}\left(\frac{c \log\log B}{n} \right)^a.\notag
\end{align}
Plugging in $M=N$, $k=m$, $S=\emptyset$, and
$$B=\max(e^{(C+2)\beta \log\log N},e^{K(C+2)^2(\log D)^2 (\log\log (DN))^2},n\log^{C+2}(N)),$$
yields the necessary result.

We prove Equation \ref{SimpleInductionEquation} by induction on $k$.  For $k=0$, the sum is at most the sum over $b=p_1\ldots p_n$ with appropriate conditions of $\frac{1}{n!}$.  Since each such $b$ can be written as such a product in at most $n!$ ways, this is at most $\Pi_n(N)$, which by Corollary \ref{maxDivDensityCor} is at most $c\left(\frac{N}{\sqrt{\log\log N}}\right)$ for some constant $c$, as desired.

For larger values of $k$, we use the inductive hypothesis and Lemma \ref{simpleInductiveStepLem} to bound $Q(n,N,D,k,S)$ by
\begin{align*}
& c N\log(N)\log^{-C-2}(M) + \frac{1}{n}\sum_{p<B} Q(n-1,N/p,D,k-1,S') \\
\leq &  c N\log(N)\log^{-C-2}(M) \\ & + \frac{1}{n} \sum_{p<B} \frac{1}{p}c\left(\frac{N}{\sqrt{\log\log(N/pB^{n-1})}}\right)\left(\frac{c\log\log B}{n-1} \right)^{k-1}\\
& + \frac{1}{n}\sum_{p<B} \frac{1}{p} c N\log(N)\log^{-C-2}(M)\sum_{a=0}^{k-2}\left(\frac{c\log\log B}{n-1} \right)^a\\
\leq &  c N\log(N)\log^{-C-2}(M) \\ & + c\left(\frac{N}{\sqrt{\log\log(N/B^{n})}}\right)\left(\frac{c\log\log B}{n} \right)^{k}\\
& + cN\log(N)\log^{-C-2}(M)\sum_{a=0}^{k-2}\left(\frac{c\log\log B}{n} \right)^{a+1}\\
\leq & c\left(\frac{N}{\sqrt{\log\log(N/B^{n})}}\right)\left(\frac{c\log\log B}{n} \right)^{k}\\
& + cN\log(N)\log^{-C-2}(M)\sum_{a=0}^{k-1}\left(\frac{c\log\log B}{n} \right)^{a}.
\end{align*}
Above we use that
$$
\frac{1}{n}\left(\frac{1}{n-1} \right)^a\sum_{p<B}\frac{1}{p} \leq c\log\log B \left(\frac{1}{n}\right)^{a+1}
$$ for all $a\leq n$ if $c$ is sufficiently large.  This completes the inductive hypothesis, proving Equation \ref{SimpleInductionEquation}, and completing the proof.
\end{proof}

We are now prepared to prove Proposition \ref{averageClassProp}
\begin{proof}
First note that we can assume that $4|D$.  This is because if that is not the case, we can split our sum up into two cases, one where none of the $p_i$ are 2, and one where one of the $p_i$ is 2.  In either case we get a sum of the same form but now can assume that $D$ is divisible by 4.  We assume this so that we can use Proposition \ref{simpleCharBoundProp}.

It is clear that the difference between the left hand side of Equation \ref{averageClassEquation} and the main term on the right hand side is
\begin{align*}
\frac{1}{|G|}\left(\sum_{\chi \in \widehat{G}\backslash \{1\}} \left(\frac{1}{n!}\sum_{S_{N,n,D}} \chi(p_1,\ldots,p_n) \right)\left(\sum_{g\in G} f(g)\chi(g) \right) \right).
\end{align*}
Using Cauchy-Schwarz we find that this is at most
$$
\frac{1}{|G|}\sqrt{|G|}|f|_2 \left(\sum_{\chi \in \widehat{G}\backslash \{1\}} \left|\frac{1}{n!}\sum_{S_{N,n,D}} \chi(p_1,\ldots,p_n) \right|^2\right)^{1/2}.
$$
We note that $|f|_2\leq \sqrt{|G|}$ and hence that $\frac{1}{|G|}\sqrt{|G|}|f|_2\leq 1$.
Bounding the character sum using Proposition \ref{simpleCharBoundProp} (using the minimal possible value of $B$), we get $O\left(\frac{N^2}{\log\log N}\right)$ times
$$
\sum_{\chi \in \widehat{G}\backslash \{1\}} O_D\left(\frac{\log\log\log N}{\log\log N} \right)^{2s}.
$$
Where above $s$ is the number of components on which $\chi$ (thought of as a product of characters of $(\Z/D\Z)^*$) is non-trivial.
Since each component of $\chi$ can either be trivial or have one of finitely many non-trivial values (each of which contributes $O_D((\log\log\log N)^2/(\log\log N)^2)$) and this can be chosen independently for each component, the inner sum is
\begin{align*}
\left( 1 + O_D\left( \frac{\log\log\log N}{\log \log N}\right)^2\right)^n-1 & = \exp\left( O_D\left( \frac{(\log\log\log N)^2}{\log\log N}\right)\right)-1 \\
& = O_D\left(\frac{(\log\log\log N)^2}{\log\log N}\right).
\end{align*}
Hence the total error is at most
$$
\frac{1}{|G|}\sqrt{|G|}\sqrt{|G|} O_D\left(\left(\frac{N^2\log\log\log^2(N)}{\log\log^2(N)}  \right)^{1/2}\right) = O_D\left( \frac{N\log\log\log(N)}{\log\log(N)}\right).
$$
\end{proof}

The proof of Proposition \ref{characterBoundProp} is along the same lines as the proof of Proposition \ref{simpleCharBoundProp}.  Again we induct on $m$.  This time, we use Lemma \ref{simpleInductiveStepLem} as our base case (when all of the $e_{i,j}$ are 0).  If some $e_{i,j}$ is non-zero, we break into cases based on whether or not $p_i$ and $p_j$ are larger than some integer $A$ (which will be some power of $\log(N)$).  If both, $p_i$ and $p_j$ are large, then fixing the remaining primes and summing over $p_i$ and $p_j$ gives a relatively small result.  Otherwise, fixing one of these primes at a small value, we are left with a sum of a similar form over the other primes.  Unfortunately, doing this will increase our $D$ by a factor of $p_i$, and may introduce characters with bad Siegel zeroes.  To counteract this, we will begin by throwing away all terms in our sum where $D\prod_i p_i$ is divisible by the modulus of the worst Siegel zero in some range, and use standard results to bound the badness of other Siegel zeroes.

We begin with some Lemmas that will allow us to bound sums of Legendre symbols of $p_i$ and $p_j$ as they vary over primes.
\begin{lem}\label{largeSieveLem}
Let $Q$ and $N$ be positive integers with $Q^2 \geq N$.  Let $a$ be a function $\{1,2,\ldots,N\}\rightarrow \C$, supported on square-free numbers.  Then we have that
$$
\sum_{\begin{subarray}{l} \chi \ \textrm{quadratic character} \\ \textrm{of modulus } p \textrm{ or } 4p, \ \leq Q\end{subarray}}\left| \sum_{n=1}^N a_n \chi(n) \right|^2 = O\left(Q\sqrt{N}||a||^2\right)
$$
where the sum is over quadratic characters whose modulus is either a prime or four times a prime and is less than or equal to $Q$, and
where $||a||^2=\sum_{n=1}^N |a_n|^2$ is the squared $L^2$ norm.
\end{lem}

Note the similarity between this and Lemma 4 of \cite{kn:HB}.

\begin{proof}
Let $M$ be the largest positive integer so that $Q^2 \leq NM^2 \leq 4Q^2$.  Let $b:\{1,2,\ldots,M^2\}\rightarrow \C$ be the function $b_{n^2}=\frac{1}{M}$ and $b=0$ on non-squares.  Let $c=a*b$ be the multiplicative convolution of $a$ and $b$.  Note that since $a$ is supported on square-free numbers and $b$ supported on squares that $||c||^2 = ||a||^2||b||^2 = ||a||^2/M$.  Applying the multiplicative large sieve inequality (see \cite{kn:IK} Theorem 7.13) to $c$ we have that
\begin{equation}\label{LSEquation}
\sum_{q\leq Q}\frac{q}{\phi(q)}\sum_{\chi \ \textrm{mod} \ q}^{\ \ \ \ \ *} \left| \sum_n c_n\chi(n) \right|^2 \leq (Q^2+NM^2-1)||c||^2.
\end{equation}
The right hand side is easily seen to be $$O(Q^2)||a||^2/M = O(Q^2||a||^2/(\sqrt{Q^2/N})) = O(Q\sqrt{N}||a||^2).$$  For the left hand side, we may note that it only becomes smaller if we remove the $\frac{q}{\phi(q)}$ or ignore the characters that are not quadratic or do not have moduli either a prime or four times a prime.  For such characters $\chi$ note that
$$
\sum_n c_n\chi(n) = \left( \sum_n a_n\chi(n)\right)\left(\sum_n b_n\chi(n)\right) = \Omega( \sum_n a_n\chi(n)).
$$
Where the last equality above follows from the fact that $\chi$ is 1 on squares not dividing its modulus, and noting that since its modulus divides four times a prime, the latter case only happens at even numbers of multiples of $p$.  Hence the left hand side of Equation \ref{LSEquation} is at least a constant multiple of
$$
\sum_{\begin{subarray}{l} \chi \ \textrm{quadratic character} \\ \textrm{of modulus } p \textrm{ or } 4p, \ \leq Q\end{subarray}}\left| \sum_{n=1}^N a_n \chi(n) \right|^2.
$$
This completes our proof.
\end{proof}

\begin{lem}\label{LegendreBoundLem}
Let $A\leq X$ be positive numbers, and let $a,b:\Z\rightarrow \C$ be functions so that $|a(n)|,|b(n)|\leq 1$ for all $n$.  We have that
$$
\left|\sum_{\begin{subarray}{l}A\leq p_1,p_2\\ p_1p_2\leq X\end{subarray}} a(p_1)b(p_2) \left(\frac{p_1}{p_2}\right)\right| = O(X\log(X) A^{-1/8}).
$$
Where the above sum is over pairs of primes $p_i$ bigger than $A$ with $p_1p_2\leq X$, and where $\left(\frac{p_1}{p_2}\right)$ is the Legendre symbol.
\end{lem}
\begin{proof}
We first bound the sum of the terms for which $p_1\leq \sqrt{X}$.

We begin by partitioning $[A,\sqrt{X}]$ into $O(A^{1/4}\log(X))$ intervals of the form $[Y,Y(1+A^{-1/4}))$.  We break up our sum based on which of these intervals $p_1$ lies in.  Once such an interval is fixed, we throw away the terms for which $p_2 \geq X/(Y(1+A^{-1/4}))$.  We note that for such terms $p_1p_2 \geq X(1+A^{-1/4})^{-1}$.  Therefore the number of such terms in our original sum is at most $O(XA^{-1/4})$, and thus throwing these away introduces an error of at most $O(XA^{-1/4})$.

The sum of the remaining terms is at most
$$
\sum_{A\leq p_2\leq X/(Y(1+A^{-1/4}))}\left| \sum_{Y\leq p_1\leq Y(1+A^{-1/4})} a(p_1)\left(\frac{p_1}{p_2}\right)\right|.
$$
By Cauchy-Schwarz, this is at most
$$
\sqrt{X/Y}\left(\sum_{A\leq p_2\leq X/(Y(1+A^{-1/4}))}\left| \sum_{Y\leq p_1\leq Y(1+A^{-1/4})} a(p_1)\left(\frac{p_1}{p_2}\right)\right|^2\right)^{1/2}.
$$
In the evaluation of the above, we may restrict the support of $a$ to primes between $Y$ and $Y(1+A^{-1/4})$.  Therefore, by Lemma \ref{largeSieveLem}, the above is at most
$$
\sqrt{X/Y}\cdot O\left(\sqrt{(X/Y)Y^{1/2}(YA^{-1/4})}\right) = O\left(XY^{-1/4}A^{-1/8} \right) = O(XA^{-3/8}).
$$
Hence summing over the $O(A^{1/4}\log(X))$ such intervals, we get a total contribution of $O(X\log(X) A^{-1/8}).$

We get a similar bound on the sum of terms for which $p_2\leq \sqrt{X}$.  Finally we need to subtract off the sum of terms where both $p_1$ and $p_2$ are at most $\sqrt{X}$.  This is
$$\sum_{A\leq p_1 \leq \sqrt{X}}\sum_{A\leq p_2\leq \sqrt{X}} a(p_1)b(p_2)\left(\frac{p_1}{p_2}\right).$$
This is at most
$$
\sum_{A\leq p_2\leq \sqrt{X}}\left| \sum_{A\leq p_1\leq \sqrt{X}} a(p_1)\left(\frac{p_1}{p_2}\right)\right|.
$$
By Cauchy-Schwarz and Lemma \ref{largeSieveLem}, this is at most
$$
\sqrt{X^{1/2}}O\left(\sqrt{X^{1/2}X^{1/4}X^{1/2}} \right) = O(X^{7/8})=O(XA^{-1/8}).
$$

Hence all of our relevant factors are $O(X\log(X) A^{-1/8})$, thus proving our bound.
\end{proof}

As mentioned above, in proving Proposition \ref{characterBoundProp}, we are going to want to deal separately with the terms in which $D\prod_i p_i$ is divisible by a particular bad Siegel zero.  In particular, for $X\leq Y$, let $q(X,Y)$ be the modulus of the Dirichlet character with the worst (closest to 1) Siegel zero of any Dirichlet character with modulus between $X$ and $Y$.  In analogy with the $Q$ defined in the proof of Proposition \ref{simpleCharBoundProp}, for integers $n,N,D,k,X,Y$ and a set $S$ of primes, we define $Q(n,N,D,k,X,Y,S)$ to be the largest possible value of
\begin{equation}\label{QformEquation}
\left|\frac{1}{n!}
\sum_{S_{N,n,D}'}
\prod_i \chi_i(p_i)
\prod_{i<j} (-1)^{\epsilon(p_i)\epsilon(p_j)d_{i,j}}\prod_{i<j} \left(\frac{p_i}{p_j}\right)^{e_{i,j}}
\right|.
\end{equation}
Above $S_{N,n,D}'$ is the subset of $S_{N,n,D}$ so that none of the $p_i$ are in $S$ and so that $q(X,Y)$ does not divide $D\prod p_i$, and
where the $\chi_i$ are Dirichlet characters of modulus dividing $D$, $e_{i,j},d_{i,j}\in \{0,1\}$, and $k$ is at most the number of indices $i$ so that one of:
\begin{itemize}
\item $e_{i,j}=1$ for some $j$ or
\item $\chi_i$ has modulus not dividing 4 or
\item $\chi_i$ has modulus exactly 4 and $d_{i,j}=0$ for all $j$.
\end{itemize}

We wish to prove an inductive bound on $Q$.  In particular we show:
\begin{lem}\label{InductiveLem}
Let $n,N,D,k,X,Y$ be as above.  Let $\beta$ be a real number so that the worst Siegel zero of a Dirichlet series of modulus at most $D$ other than $q(X,Y)$ is at most $1-\beta^{-1}$.  Let $M,A,B,C$ be integers so that
$$
B > \max(e^{(C+2)\beta \log\log M},e^{K(C+2)^2(\log D)^2 (\log\log (DM))^2},n\log^{C+2}(M),A)
$$
for a sufficiently large constant $K$.
Then for $S$ a set of primes $\leq A$, we have that $Q(n,N,D,k,X,Y,S)$ is at most the maximum of
$$
N\left( O\left(\frac{\log\log B}{n}\right)^k + O(\log(N)\log^{-C-2}(M))\sum_{a=0}^{k-1}O\left(\frac{\log\log B}{n}\right)^a\right)
$$
and
\begin{align*}
O & (N\log^2(N) A^{-1/8}) + \frac{2}{n}\sum_{p<A} Q(n-1,N/p,Dp,k-1,X,Y,S\cup\{p\})\\
& + \frac{1}{n(n-1)}\sum_{p_1,p_2<A}Q(n-2,N/p_1p_2,Dp_1p_2,k-2,X,Y,S\cup\{p_1,p_2\}).
\end{align*}
\end{lem}

\begin{proof}
We consider a sum of the form given in Equation \ref{QformEquation}.  If all of the $e_{i,j}$ are 0, we have a form of the type handled in the proof of Proposition \ref{simpleCharBoundProp}, and our sum is bounded by the first of our two expressions by Equation \eqref{SimpleInductionEquation}.

Otherwise, some $e_{i,j}$ is 1.  Without loss of generality, this is $e_{n-1,n}$.  We can also assume that $d_{n-1,n}=0$ since adding or removing the appropriate term is equivalent to reversing the Legendre symbol.  We split our sum into parts based on which of $p_{n-1},p_n$ are at least $A$.  In particular we take the sum of terms with both at least $A$, plus the sum of terms where $p_{n-1}<A$ plus the sum of terms with $p_n<A$ minus the sum of terms with both less than $A$.

First, consider the case where $p_{n-1},p_{n}\geq A$.  Fixing the values of $p_1,\ldots,p_{n-2}$, and letting $P=\prod_{i=1}^{n-2} p_i$, we consider the remaining sum over $p_{n-1}$ and $p_n$.  We have
$$
\frac{\pm 1}{n!} \sum_{\begin{subarray}{l} A\leq p_{n-1},p_n \\ p_{n-1}\neq p_n \\ (p_i,DP)=1 \\ Q \not | DPp_{n-1}p_n \\ p_{n-1}p_n \leq N/P \end{subarray}} a(p_{n-1})b(p_n)\left(\frac{p_{n-1}}{p_n}\right).
$$
Where $a,b$ are some functions $\Z\rightarrow \C$ so that $|a(x)|,|b(x)|\leq 1$ for all $x$.  We note that the condition that $(p_i,DP)=1$ can be expressed by setting $a$ and $b$ equal to 0 for some appropriate set of primes.  We note that the condition that $q(X,Y)$ not divide $DPp_{n-1}p_n$ is only relevant if $DP$ is missing only one or two primes of $q(X,Y)$.  In the former case, it is equivalent to making one more value illegal for the $p_i$.  In the latter case it eliminates at most two terms.  The condition that the $p_i$ are distinct removes at most $\sqrt{N/P}$ terms from our sum.  Therefore, perhaps after setting $a$ and $b$ to $0$ on some set of primes, the above is
$$
\frac{\pm 1}{n!}\left(O(\sqrt{N/P})+ \sum_{\begin{subarray}{l} A\leq p_{n-1},p_n  \\ p_{n-1}p_n \leq N/P \end{subarray}} a(p_{n-1})b(p_n)\left(\frac{p_{n-1}}{p_n}\right)\right).
$$
By Lemma \ref{LegendreBoundLem}, this is at most
$$
\frac{1}{n!}O(N/P \log(N) A^{-1/8}).
$$
Now for each $P\leq N$, it can be written in at most $(n-2)!$ ways, hence the sum over all $p_{n-1},p_n\geq A$ is at most
$$
\sum_{P=1}^N O(N/P \log(N) A^{-1/8}) = O(N\log^2(N) A^{-1/8}).
$$

Next, we consider the case where $p_n<A$.  We deal with this case by setting $p_n$ to each possible value of size at most $A$ individually.  It is easy to check that after setting $p_n$ to such a value $p$, the sum over the remaining $p_i$ is $1/n$ times a sum of the form bounded by $Q(n-1,N,Dp,k-1,X,Y,S\cup\{p\})$.   Hence the sum over all terms with $p_n <A$ is at most
$$
\frac{1}{n}\sum_{p<A} Q(n-1,N/p,Dp,k-1,X,Y,S\cup\{p\}).
$$

The sum of the terms with $p_{n-1}<A$ has the same bound, and the sum of terms with both less than $A$ is similarly seen to be at most
$$
\frac{1}{n(n-1)}\sum_{p_1,p_2<A}Q(n-2,N/p_1p_2,Dp_1p_2,k-2,X,Y,S\cup\{p_1,p_2\}).
$$
\end{proof}

We now use Lemma \ref{InductiveLem} to prove an inductive bound on $Q$.
\begin{lem}\label{QBoundLem}
Let $n,N,D,k,X,Y,S,M,A,B,C,\beta$ be as above.  Assume furthermore that $Y\geq D A^n$,
$$
B > \max(e^{(C+2)\beta \log\log M},e^{K(C+2)^2(\log Y)^2 (\log\log (YM))^2},n\log^{C+2}(M),A),
$$
and that $S$ contains only elements of size at most $A$.  Let $L=n-k$, then $Q(n,N,D,k,X,Y,S)$ is at most
\begin{align*}
N\Bigg( & O\left(\frac{\log\log B}{L}\right)^k  \\
 & + O\left( \log^2(N) A^{-1/8} + \log(N)\log^{-C-2}(M)\right)\sum_{a=0}^{k-1}O\left(\frac{\log\log B}{L}\right)^a\Bigg).
\end{align*}
\end{lem}

Note that we will wish to apply this Lemma with $n$ about $\log\log N$, $D$ a constant, $A$ polylog $N$, $X$ polylog $N$, $M=N$, $Y=DA^n$, and $B$ its minimum possible value.

\begin{proof}
We proceed by induction on $k$. In particular, we show that for a sufficiently large constant $c$ that $Q(n,N,D,k,X,Y,S)$ is at most
\begin{align*}
c N\Bigg( & \left(\frac{c\log\log B}{L}\right)^k  \\
 & + \left( \log^2(N) A^{-1/8} + \log(N)\log^{-C-2}(M)\right)\sum_{a=0}^{k-1}\left(\frac{c \log\log B}{L}\right)^a\Bigg).
\end{align*}

 We bound $Q$ inductively by Lemma \ref{InductiveLem}. Our base case is when $Q$ is equal to
$$
N\left( O\left(\frac{\log\log B}{n}\right)^k + O(\log(N)\log^{-C-2}(M))\sum_{a=0}^{k-1}O\left(\frac{\log\log B}{n}\right)^a\right)
$$
(which must happen if $k=0$).  In this case, our desired bound holds assuming that $c$ is sufficiently large.

Otherwise, $Q(n,N,D,k,X,Y,S)$ is bounded by
\begin{align*}
O & (N\log^2(N) A^{-1/8}) + \frac{2}{n}\sum_{p<A} Q(n-1,N/p,Dp,k-1,X,Y,S\cup\{p\})\\
& + \frac{1}{n(n-1)}\sum_{p_1,p_2<A}Q(n-2,N/p_1p_2,Dp_1p_2,k-2,X,Y,S\cup\{p_1,p_2\}).
\end{align*}
Notice that the parameters of $Q$ in the above also satisfy our hypothesis, so we may bound them inductively.  Note also that for the above values of $Q$ that the value of $L$ is the same.  Letting $U=\frac{c\log\log B}{L}$ and $$E = c\left(\log^2(N) A^{-1/8} + \log(N)\log^{-C-2}(M)\right),$$ then for $c$ sufficiently large the above is easily seen to be at most
\begin{align*}
& N\left(E + \frac{U}{2}\left(U^{k-1} + E\sum_{a=0}^{k-2}U^a \right) + \frac{U^2}{2}\left(U^{k-2} + E\sum_{a=0}^{k-3}U^a \right) \right)\\
\leq & N\left( U^k + E\sum_{a=0}^{k-1} U^a \right).
\end{align*}
This completes our inductive step and finishes the proof.
\end{proof}

We are finally prepared to prove Proposition \ref{characterBoundProp}.
\begin{proof}
The basic idea will be to compare the sum in question to the quantity $Q(n,N,D,k,X,Y,\emptyset)$  for appropriate settings of the parameters.  We begin by fixing the constant $c$ in the Proposition statement.  We let $C$ be a constant large enough that $c^n > \log^{-C}(N)$ (recall that $n$ was $O(\log\log N)$).  We set $A$ to $\log^{8C+16}(N)$, $X$ to $\log^C(N)$ and $Y$ to $DA^n= \exp(O_D(C(\log\log N)^2))$.  We let $M=N$.

We note that $\beta$ comes from either the worst Siegel zero of modulus less that $X$, or the second worst Siegel zero of modulus less than $Y$.  By Theorem 5.28 of \cite{kn:IK}, $\beta$ is at most $O_\epsilon(X^\epsilon)$ in the former case, and at most $O(\log(Y))$ in the latter case.  Hence (changing $\epsilon$ by a factor of $C$), we have unconditionally, that $\beta = O_\epsilon(\log^\epsilon (N))$ for any $\epsilon>0$.  We next let
$$
B = \max(e^{(C+2)\beta \log\log M},e^{K(C+2)^2(\log Y)^2 (\log\log (YM))^2},n\log^{C+2}(M),A).
$$
Hence for sufficiently large $N$ (in terms of $\epsilon$ and $D$),
$$
\log\log B < \epsilon \log\log N.
$$

Finally we pick $k$ so that $n/2 \geq k \geq m/2$.  Thus $L=n-k > n/2 = \Omega(\log\log N)$.  Noting that we satisfy the hypothesis of Lemma \ref{InductiveLem}, we have that for $N$ sufficiently large relative to $\epsilon$ and $D$ that $Q(n,N,D,k,X,Y,\emptyset)$ is at most
$$
N\left( O(\epsilon)^{m/2} +O\left( \log^2(N) \log^{-C-2}(N) + \log(N)\log^{-C-1}(N)\right)\sum_{a=0}^{k}O(\epsilon)^a\right).
$$
If $\epsilon$ is small enough that the term $O(\epsilon)$ is at most $1/2$, this is at most
$$
N\left( O(\epsilon)^{m/2} + \log^{-C}(N)\right).
$$
If additionally the $O(\epsilon)$ term is less than $c^2$, this is
$$
O(Nc^m).
$$
Hence for $N$ sufficiently large relative to $c$ and $D$,
$$
Q(n,N,D,k,X,Y,\emptyset) = O(Nc^m).
$$
Therefore unequivocally,
$$
Q(n,N,D,k,X,Y,\emptyset) = O_{c,D}(Nc^m).
$$

Finally, we note that the difference between $Q(n,N,D,k,X,Y,\emptyset)$ and the term that we are trying to bound is exactly the sum over such terms where $p_1\cdots p_n$ is divisible by $\frac{q(X,Y)}{gcd(q(X,Y),D)}.$  Since $q(X,Y)\geq X$, there are only $O_D(N \log^{-C}(N))$ such products.  Since each product can be obtained in at most $n!$ ways, each contributing at most $\frac{1}{n!}$, this difference is at most $O_D(N \log^{-C}(N))=O(Nc^m)$.  Therefore the thing we wish to bound is $O_{c,D}(Nc^m).$
\end{proof}

\section{Average Sizes of Selmer Groups}\label{MomentSec}

Here we use the results from the previous section to prove the following Proposition:
\begin{prop}\label{expectedValueProposition}
Let $E$ be an elliptic curve satisfying the conditions of Theorem \ref{mainThmSel} (and in particular by Theorem \ref{SDThm}, for any $E$ with full 2-torsion defined over $\Q$ and no cyclic 4-isogeny defined over $\Q$).  Let $S$ be a finite set of places containing $2,\infty$ and all of the places where $E$ has bad reduction.  Let $x$ be either $-1$ or a power of 2.  Let $\omega(m)$ denote the number of prime factors of $m$.  Say that $(m,S)=1$ if $m$ is an integer not divisible by any of the finite places in $S$.  For positive integers $N$, let $\mathcal{S}_N$ denote the set of integers $b\leq N$ squarefree with $|\omega(b)-\log\log N|\leq (\log\log N)^{3/4},$ and $ (b,S)=1$.  Then
$$
\lim_{N\rightarrow\infty}\frac{\sum_{\mathcal{S}_N}x^{\dim(S_2(E_b))}}{|\mathcal{S}_N|} = \sum_n x^n\alpha_n.
$$
\end{prop}

This says that the $k^{th}$ moment of $|S_2(E_b)|$ averaged over $b\leq N$ with $$|\omega(b)-\log\log N|\leq (\log\log N)^{3/4}$$ is what you would expect given Theorem \ref{SDThm}. Furthermore, Proposition \ref{expectedValueProposition} says that averaged over the same set of $b$'s, that the rank of the Selmer group is odd half of the time.  The latter part of the Proposition follows from Lemma \ref{parityLem}, which we prove now:

\begin{proof}[Proof of Lemma \ref{parityLem}]
First we replace $E$ by a twist so that $c_i-c_j$ are pairwise relatively prime integers.  It is now the case that $E$ has everywhere good or multiplicative reduction, and we are now concerned with $\dim(S_2(E_{db}))$ for some constant $d|D$.  By \cite{kn:MR} Theorem 2.7, and \cite{kn:K} Corollary 1 we have that $\dim(S_2(E_{bd})) \equiv \dim(S_2(E)) \pmod{2}$ if and only if $(-1)^{x}\chi_{bd}(-N)=1$ where $x=\omega(d)$, $N$ is the product of the primes not dividing $d$ at which $E$ has bad reduction, and $\chi_{bd}$ is the quadratic character corresponding to the extension $\Q(\sqrt{bd})$.  From this, the Lemma follows immediately.
\end{proof}

In order to prove the rest of Proposition \ref{expectedValueProposition}, we will need a concrete description of the Selmer groups of twists of $E$.  We follow the treatment given in \cite{kn:SD}.  Let $b=p_1\cdots p_n$ where $p_i$ are distinct primes relatively prime to $S$ (we leave which primes unspecified for now).  Let $B=S\cup \{p_1,\ldots,p_n\}$.  For $\nu\in B$ let $V_\nu$ be the subspace of $(u_1,u_2,u_3)\in (\Q_\nu^*/(\Q_\nu^*)^2)^3$ so that $u_1u_2u_3=1$.  Note that $V_\nu$ has a symplectic form given by $(u_1,u_2,u_3)\cdot (v_1,v_2,v_3) = \prod_{i=1}^3 (u_i,v_i)_\nu$, where $(u_i,v_i)_\nu$ is the Hilbert Symbol.  Let $V=\prod_{\nu\in B} V_\nu$ be a symplectic $\mathbb{F}_2$-vector space of dimension $2M$.

There are two important Lagrangian subspaces of $V$.  The first, which we call $U$, is the image in $V$ of $(\Z_B^*/(\Z_B^*)^2)^3_1$.  The other, which we call $W$, is given as the product of $W_\nu$ over $\nu\in B$, where $W_\nu$ consists of points of the form $(x-bc_1,x-bc_2,x-bc_3)$ for $(x,y)\in E_b$.  Note that we can write $W=W_S\times W_b$ where $W_S=\prod_{\nu\in S} W_\nu$ and $W_b=\prod_{\nu|b} W_\nu$.  The Selmer group is given by
$$
S_2(E_b) = U\cap W.
$$

As written $U,W$ and $V$ all depend on the primes dividing $b$.  Fortunately, as we will see, there are natural spaces $U',W'$ that depend very little on $b$ with convenient isomorphisms to $U$ and $W$.  It would also be possible to similarly parameterize $V$, but this will prove to be unnecessary as we intend to compute the size of the intersection of $U$ and $W$ solely in terms of the restriction of the symplectic pairing on $V$ to $U\times W$.

Let $U'$ be the $\mathbb{F}_2$-vector space generated by the symbols $\nu,\nu'$ for $\nu\in S$ and $p_i,p_i'$ for $1\leq i \leq n$.  There is an isomorphism $f:U'\rightarrow U$ given by $f(\infty)=(-1,-1,1),f(\infty')=(1,-1,-1),f(p)=(p,p,1),f(p')=(1,p,p)$.

Note also that $W_{p_i}$ is generated by $((c_1-c_2)(c_1-c_3),b(c_1-c_2),b(c_1-c_3))$ and $(b(c_3-c_1),b(c_3-c_2),(c_3-c_1)(c_3-c_2))$.  If we define $W'$ to be the $\mathbb{F}_2$-vector space generated by the symbols $p_i,p_i'$ for $1\leq i \leq n$, then there is an isomorphism $g:W'\rightarrow W_b$ given by $g(p_i) = ((c_1-c_2)(c_1-c_3),b(c_1-c_2),b(c_1-c_3))\in W_{p_i}$ and $g(p_i') = (b(c_3-c_1),b(c_3-c_2),(c_3-c_1)(c_3-c_2)) \in W_{p_i}$.

Let $G=\prod_{\nu\in S} \mathfrak{o}_\nu^*/(\mathfrak{o}_\nu^*)^2$ (here $\mathfrak{o}_\nu^*$ are the units in the ring of integers of $k_\nu$).  Note that $W_S$ is determined by the restriction of $b$ to $G$.  So for $c\in G$ let $W_{S,c}$ be $W_S$ for such $b$.  Let $W'_c = W_{S,c}\times W'$.  Then we have a natural map $g_c:W_c'\rightarrow V$ that is an isomorphism between $W_c'$ and $W$ if $b$ restricts to $c$.

We are now ready to prove Proposition \ref{expectedValueProposition}.
\begin{proof}
For $x=-1$ this Proposition just says that the parity is odd half of the time, which follows from Lemma \ref{parityLem}.  For $x=2^k$ this says something about the expected value of $|S_2(E_b)|^k$.  For $x=2^k$ we will show that for each $n\in (\log\log N - (\log\log N)^{3/4},\log\log N + (\log\log N)^{3/4})$ that
\begin{align*}
&\sum_{S_{N,n,D}}|S_2(E_b)|^k = |S_{N,n,D}|\left(\sum_m \alpha_m (2^k)^m+\delta(n,N)\right) + O_{E,k}\left(\frac{N(\log\log\log N)^2}{\log\log N}\right).
\end{align*}
Where $\delta(n,N)$ is some function so that $\lim_{N\rightarrow\infty}\delta(n,N)=0$.  Summing over $n$ and noting that there are $\Omega(N)$ values of $b\leq N$ square-free with $|\omega(b)-\log\log N|<(\log\log N)^{3/4}$, and $(b,S)=1$ gives us our desired result.

In order to do this we need to better understand $|S_2(E_b)| = |U\cap W|$.  For $v\in V$ we have since $U$ is Lagrangian of size $2^M$,
\begin{align*}
\frac{1}{2^M}\sum_{u\in U} (-1)^{u\cdot v} & = \begin{cases} 1 \ \textrm{if} \ v\in U^\perp \\ 0 \ \textrm{else} \end{cases}\\
& = \begin{cases} 1 \ \textrm{if} \ v\in U \\ 0 \ \textrm{else} \end{cases}.
\end{align*}
Hence
\begin{align*}
|S_2(E_b)| & =|U\cap W|\\
& = \#\{w\in W:w\in U\}\\
& = \sum_{w\in W} \frac{1}{2^M}\sum_{u\in U} (-1)^{u\cdot w}\\
& = \frac{1}{2^M}\sum_{u\in U, w\in W} (-1)^{u\cdot w}\\
& = \frac{1}{2^M}\sum_{u\in U',w\in W_b'} (-1)^{f(u)\cdot g_b(w)}\\
\end{align*}
If we extend $f$ and $g_c$ to $f^k:(U')^k\rightarrow U^k$, $g_c^k:(W_c')^k\rightarrow V^k$, and extend the inner product on $V$ to an inner product on $V^k$, we have that
$$
|S_2(E_b)|^k = \frac{1}{2^{kM}}\sum_{\begin{subarray}{l}u \in (U')^k, w\in (W_b')^k \end{subarray}}(-1)^{f^k(u)\cdot g_b^k(w)},
$$
and therefore that
\begin{equation}\label{SelmerSizeEquation}
|S_2(E_b)|^k = \frac{1}{2^{kM}|G|}\sum_{\begin{subarray}{l} c\in G,\chi \in \widehat{G} \\ u \in (U')^k, w\in (W_c')^k \end{subarray}}\chi(bc^{-1})(-1)^{f^k(u)\cdot g_c^k(w)}.
\end{equation}
Notice that once we fix values of $c,\chi,u,w$ in Equation (\ref{SelmerSizeEquation}), the summand (when treated as a function of $p_1,\ldots,p_n$) is of the same form as the ``characters'' studied in Section \ref{charBoundSec}.

We want to take the sum over $S_{N,n,D}$ of $|S_2(E_b)|^k$.  If we let $D$ be 8 times the product of the finite odd primes in $S$, we note that each such $b$ can be expressed exactly $n!$ ways as a product $b=p_1\cdots p_n$ with $p_i$ distinct, $(p_i,D)=1$.  Therefore this sum equals
\begin{align*}
\frac{1}{n!}
\sum_{S_{N,n,D}}\frac{1}{2^{kM}|G|}\sum_{\begin{subarray}{l} c\in G,\chi \in \widehat{G} \\ u \in (U')^k, w\in (W_c')^k \end{subarray}}\prod_i\chi(p_i)\bar\chi(c)(-1)^{f^k(u)\cdot g_c^k(w)}.
\end{align*}
Interchanging the order of summation gives us
\begin{align*}
\frac{1}{2^{kM}|G|}\sum_{S_{N,n,D}}\frac{\bar\chi(c)}{n!}\sum_{\begin{subarray}{l} p_1,\ldots,p_n \\ \textrm{distinct primes} \\ (D,p_i)=1 \\ \prod_i p_i \leq N \end{subarray}}\left(\prod_i\chi(p_i)\right)(-1)^{f^k(u)\cdot g_c^k(w)}.
\end{align*}
Now the inner sum is exactly of the form studied in Proposition \ref{characterBoundProp}.

We first wish to bound the contribution from terms where this inner sum has terms of the form $\left(\frac{p_i}{p_j}\right)$, or in the terminology of Proposition \ref{characterBoundProp} for which not all of the $e_{i,j}$ are 0.  In order to do this, we will need to determine how many of these terms there are and how large their values of $m$ are.  Notice that terms of the form $\left(\frac{p_i}{p_j}\right)$ show up here when we are evaluating the Hilbert symbols of the form $(p,b(c_a-c_b))_p,(p,b(c_a-c_b))_q,(q,b(c_a-c_b))_p,(q,b(c_a-c_b))_q$ and in no other places.

Let $U_i\subset U'$ be the subspace generated by $p_i=(p_i,p_i,1)$ and $p_i'=(1,p_i,p_i)$.  For $u\in U'$ let $u_i$ be its component in $U_i$ in the obvious way.  Let $W_i\subset W'$ be $W_{p_i}$.  For $w\in W_c'$ let $w_i$ be its component in $W_i$.  It is not hard to see that the power of $\left(\frac{p_i}{p_j}\right)$ appearing in $(-1)^{f^k(u)\cdot g_c^k(w)}$ depends only on the projections of $u$ and $w$ onto $U_i\times U_j$ and $W_i\times W_j$, respectively.  Our analysis of these exponents will be simplified considerably, by noting that the $U_i$ and $W_i$ have convenient isomorphisms to fixed spaces which we call $U_0$ and $W_0$.  In particular, let $U_0$ be the $\mathbb{F}_2$-vector space with formal generators $p$ and $p'$.  We have a natural isomorphism between $U_0$ and $U_i$ sending $p$ to $p_i$ and $p'$ to $p_i'$.  We will hence often think of $u_i$ as an element of $U_0$.  Similarly let $W_0$ be the $\mathbb{F}_2$-vector space with formal generators $((c_1-c_2)(c_1-c_3),b(c_1-c_2),b(c_1-c_3))$ and $(b(c_3-c_1),b(c_3-c_2),(c_3-c_1)(c_3-c_2))$.  We similarly have natural isomorphisms between $W_i$ and $W_0$ and will often consider $w_i$ as an element of $W_0$ instead of $W_i$.

Additionally, we have a bilinear form $U_0\times W_0\rightarrow \mathbb{F}_2$ defined by:
\begin{align*}
p\cdot & ((c_1-c_2)(c_1-c_3),b(c_1-c_2),b(c_1-c_3)) \\
& = p'\cdot (b(c_3-c_1),b(c_3-c_2),(c_3-c_1)(c_3-c_2))\\
& = 1\\
p'\cdot & ((c_1-c_2)(c_1-c_3),b(c_1-c_2),b(c_1-c_3))\\
&= p \cdot (b(c_3-c_1),b(c_3-c_2),(c_3-c_1)(c_3-c_2))\\
& = 0.
\end{align*}
Notice that if $u\in U'$ and $w\in W_c'$, then the exponent of $\left(\frac{p_i}{p_j}\right)$ that appears in $(-1)^{f(u)\cdot g_c(w)}$ is $(u_i+u_j)\cdot (w_i+w_j)$.  Similarly, if $u\in (U')^k, w\in (W_c')^k$, the exponent of $\left(\frac{p_i}{p_j}\right)$ that appears in $(-1)^{f^k(u)\cdot g_c^k(w)}$ is $(u_i+u_j)\cdot (w_i+w_j)$, where $u_*,w_*$ are thought of as elements of $U_0^k$ and $W_0^k$, and the inner product is extended to $U_0^k\times W_0^k$ as $(x_1,\ldots,x_k)\cdot (y_1,\ldots,y_k) = \sum_{i=1}^k x_i\cdot y_i$.

Let $T=U_0^k\times W_0^k.$  We define a symplectic form on $T$ by $\dotp{(u,w)}{ (u',w')} = u\cdot w' + u'\cdot w$.  Also define a quadratic form $q$ on $T$ by $q(u,w)=u\cdot w$. We claim that given some sequence of elements, $t_x=(u_x,w_x)\in T$ for $x\in I$, that $(u_x+u_y)\cdot (w_x+w_y)=0$ for all pairs $x,y\in I$ only if all of the $t_x$ lie in a translate of a Lagrangian subspace of $T$ under the symplectic form $\dotp{-}{-}$.  To show this, we note that for $t=(u,w),t'=(u',w')$ that $(u+u')\cdot (w+w') = \dotp{t}{ t'} + q(t)+q(t')$.  We need to show that for all $x,y,z \in I$ that $\dotp{(t_x+t_y)}{ (t_x+t_z)}=0$.  This is true because
\begin{align*}
&\dotp{(t_x  +t_y) }{ (t_x+t_z)} \\ & = \dotp{t_x}{ t_x} + \dotp{t_x}{ t_z} + \dotp{t_y}{ t_x} +\dotp{t_y}{ t_z}\\
& = \dotp{t_x}{ t_z} + \dotp{t_y}{ t_x} +\dotp{t_y}{ t_z}\\
& = \dotp{t_x}{ t_z} + \dotp{t_y}{ t_x} +\dotp{t_y}{ t_z}+ 2q(t_x)+2q(t_y)+2q(t_z)\\
& = (\dotp{t_y}{ t_x} + q(t_x)+q(t_y)) + (\dotp{t_x}{ t_z} + q(t_x)+q(t_z)) + (\dotp{t_y}{ t_z} + q(t_y)+q(t_z))\\
& = 0.
\end{align*}

Suppose that we have some $u=(u_1,\ldots,u_n)\in \prod_{i=1}^nU_i^k$ and $w=(w_1,\ldots,w_n)\in \prod_{i=1}^n W_i^k$, and suppose that we have a set of $\ell$ indices in $\{1,2,\ldots,n\}$, which we call \emph{active} indices,  so that $(-1)^{f^k(u)\cdot g^k(w)}$ has terms of the form $\left(\frac{p_i}{p_j}\right)$ only if $i,j$ are both active, and suppose furthermore that each active index shows up as either $i$ or $j$ in at least one such term.  Let $t_i=(u_i,w_i)\in T$ (where we have identified $u_i$ and $w_i$ as elements of $U_0^k$ and $W_0^k$, respectively).  We claim that $t_i$ takes fewer than $4^{k}$ different values on non-active indices, $i$.  We note that our notion of active indices is similar to the notion in \cite{kn:HB} of linked indices.

Since $\dotp{t_i}{ t_j} + q(t_i)+q(t_j)=0$ for any two non-active indices $t_i$ and $t_j$, all of these must lie in a translate of some Lagrangian subspace of $T$.  Therefore $t_i$ can take at most $4^{k}$ values on non-active indices.  Suppose for sake of contradiction that all of these values are actually assumed by some non-active index.  Then consider $t_j$ for $j$ an active index.  The $t_i$ for $i$ either non-active or equal to $j$ must similarly lie in a translate of a Lagrangian subspace.  Since such a space is already determined by the non-active indices and since all elements of this affine subspace are already occupied, $t_j$ must equal $t_i$ for some non-active $i$. But this means that every $t_j$ is assumed by some non-active index which implies that no terms of the form $\left(\frac{p_i}{p_j}\right)$ survive, yielding a contradiction.

Now consider the number of such $u,w$ so that there are $\ell\geq 1$ active indices.  Once we fix the values $t_i$ that are allowed to be taken by the non-active indices (which can only be done in finitely many ways), there are $\binom{n}{\ell}$ ways to choose the active indices, at most $2^{k}-1$ ways to pick $t_i$ for each non-active index, and at most $2^{2k}$ ways for each active index.  Hence the total number of such $u$,$w$ with exactly $\ell$ active indices is
$$
O\left(\binom{n}{\ell}\left(4^{k}-1\right)^{n-\ell}\left(4^{2k}\right)^\ell\right).
$$
By Proposition \ref{characterBoundProp}, the value of the inner sum for such a $(u,w)$ is at most $O_{E,k}\left(N (2^{-2k-1})^\ell\right)$.  Hence summing over all $\ell>0$ and recalling the $2^{-Mk}$ out front we get a contribution of at most
\begin{align*}
N4^{-nk}O_{E,k}\left(\sum_\ell \binom{n}{\ell}\left(4^{k}-1\right)^{n-\ell}\left( \frac{1}{2}\right)^\ell \right) & = N4^{-nk}O_{E,k}\left( (4^k-1/2)^n\right) \\ & = NO_{E,k}\left( (1-4^{-k-1})^n\right) \\ & = NO_{E,k}\left((\log N)^{-4^{-k-2}}\right).
\end{align*}

Therefore, we may safely ignore all of the terms in which a $\left(\frac{p_i}{p_j}\right)$ shows up.  This is our analogue of Lemma 6 in \cite{kn:HB}.

Notice that by the above analysis, that the number of remaining terms must be $O_{k,E}(2^{Mk})$.  Additionally, for these terms we may apply Proposition \ref{averageClassProp}. Therefore each term, up to an error of $O_E\left( \frac{(\log\log\log N)^2}{\log\log N}\right)$, equals $|S_{N,n,D}|$ times the average of its summand over all possible conjugacy classes of $p_1,\ldots,p_n$ modulo $4D$. Since there are $O_{k,E}(2^{Mk})$ such terms, and since there is an outer factor of $2^{-kM}$ we reach two conclusions. Firstly, the sum in question is bounded by $O_{k,E}(|S_{N,n,D}|)$. Secondly, $1/n!$ times the sum over $S_{N,n,D}$ of $|S_2(E_b)|^k$ is, to within an error of $O_{E,k}\left( \frac{(\log\log\log N)^2}{\log\log N}\right)$ equal to $|S_{N,n,D}|$ times the average over $b=p_1\cdots p_n$ over all possible values of $p_i$ modulo $4D$ and Legendre symbols $\left(\frac{p_i}{p_j}\right)$ of $|S_2(E_b)|^k$. By definition, this latter average is simply
$$
\sum_d \pi_d(n)2^{kd}.
$$
Using the fact that this is bounded for $k+1$ independently of $n$, we find that $\pi_d(n) = O_{k,E}(2^{-(k+1)d})$.  In order to complete the proof of our Proposition, we need to show that
$$
\lim_{n\rightarrow\infty} \sum_d(\pi_d(n)-\alpha_d)2^{kd} = 0.
$$
But this follows from the fact that
$$
\sum_{d>X} (\pi_d(n)-\alpha_d)2^{kd} = O_{E,k}\left(\sum_{d>X} 2^{-d}\right) = O_{E,k}(2^{-X})
$$
and that $\pi_d(n)\rightarrow \alpha_d$ for all $d$ by assumption.
\end{proof}

\section{From Sizes to Ranks}\label{toRankSec}

In this Section, we turn Proposition \ref{expectedValueProposition} into a proof of Theorem \ref{mainThmSel}.  This Section is analogous to Section 8 of \cite{kn:HB}, although our techniques are significantly different. We begin by doing some computations with the $\alpha_i$.

Note that
$$
\alpha_{n+2} = \left(\frac{1}{\prod_{j=0}^\infty (1+2^{-j})} \right)2^{-\binom{n}{2}}\prod_{j=1}^n (1-2^{-j})^{-1}.
$$
Now $\prod_{j=1}^n (1-2^{-j})^{-1}$ is the sum over partitions, $P$, into parts of size at most $n$ of $2^{-|P|}$.  Equivalently, taking the transpose, it is the sum over partitions $P$ with at most $n$ parts of $2^{-|P|}$.  Multiplying by $2^{-\binom{n}{2}}$, we get the sum over partitions $P$ with $n$ distinct parts (possibly a part of size 0) of $2^{-|P|}$.  Therefore, we have that
$$
F(x) = \sum_{n=0}^\infty \alpha_n x^n = \frac{x^2\prod_{j=0}^\infty (1+2^{-j}x)}{\prod_{j=0}^\infty (1+2^{-j})}.
$$
Since the $x^{d+2}$ coefficient of $F(x)$ is also the sum over partitions, $P$ into exactly $d$ distinct parts (perhaps one of which is 0) of $2{-|P|}$ divided by $\prod_{j=0}^\infty (1+2^{-j})$.  This implies in particular that $\sum_{n=0}^\infty \alpha_n$ equals 1 as it should.

Let $T_{N}$ be the set of square-free $b\leq N$ with $(b,D)=1$, and $|\omega(b)-\log\log N|<(\log\log N)^{3/4}$. Let $C_d(N)$ be
$$
\frac{\#\{b\in T_N : \dim(S_2(E_b))=d\}}{|T_{N}|}.
$$
Let $C(N) = (C_0(N),C_1(N),\ldots)\in [0,1]^\omega$.  Theorem \ref{mainThmSel} is equivalent to showing that
$$
\lim_{N\rightarrow\infty} C(N) = (\alpha_0,\alpha_1,\ldots).
$$

\begin{lem}\label{limitLem}
Suppose that some subsequence of the $C(N)$ converges to some sequence $(\beta_0,\beta_1,\ldots)\in [0,1]^\omega$ in the product topology.  Let $G(x) = \sum_n \beta_n x^n$.  Then $G(x)$ has infinite radius of convergence and $F(x)=G(x)$ for $x=-1$ or $x$ equals a power of 2.  Also $\beta_0=\beta_1=0$.
\end{lem}

This Lemma says that if the $C(N)$ have some limit that the naive attempt to compute moments of the Selmer groups from this limit would succeed.

\begin{proof}
The last claim follows from the fact that since $E_b$ has full 2-torsion, its 2-Selmer group always has rank at least 2.
Notice that $\sum_d C_d(N)x^d$ is equal to the average size of $x^{\dim(S_2(E_b))}$ over $b\leq N$ square-free, relatively prime to $D$ with $|\omega(b)-\log\log N|<(\log\log N)^{3/4}$.  This has limit $F(x)$ as $N\rightarrow \infty$ by Proposition \ref{expectedValueProposition} if $x$ is $-1$ or a power of 2.  In particular it is bounded.  Therefore there exists an $R_k$ so that
$$
\sum_d C_d(N)2^{kd} \leq R_k
$$
for all $N$.  Therefore $C_d(N) \leq R_k 2^{-kd}$ for all $d,N$.  Therefore $\beta_d \leq R_k 2^{-kd}$.  Therefore $G$ has infinite radius of convergence.

Furthermore if we pick a subsequence, $N_i\rightarrow \infty$ so that $C_d(N_i)\rightarrow \beta_d$ for all $d$, we have that
\begin{align*}
F(2^k) & = \lim_{i\rightarrow \infty} \sum_d C_d(N_i) 2^{dk} \\
& = \lim_{i\rightarrow \infty} \sum_{d\leq X} C_d(N_i) 2^{dk} + O\left(\sum_{d>X} R_{k+1} 2^{-d}\right)\\
& = \lim_{i\rightarrow \infty} \sum_{d\leq X} C_d(N_i) 2^{dk} + O(R_{k+1}2^{-X}) \\
& = \sum_{d\leq X} \beta_d 2^{dk} + O(R_{k+1}2^{-X}).
\end{align*}
So
$$
\lim_{X\rightarrow\infty} \sum_{d\leq X} \beta_d 2^{dk} = F(2^k).
$$
Thus $G(2^k)=F(2^k)$.  For $x=-1$ the argument is similar but comes from the equidistribution of parity rather than expectation of size.
\end{proof}

\begin{lem}\label{taylorLem}
Suppose that $G(x)=\sum_n \beta_n x^n$ is a Taylor series with infinite radius of convergence.  Suppose also that $\beta_n\in [0,1]$ for all $n$ and that $G(x)=F(x)$ for $x$ equal to $-1$ or a power of 2.  Suppose also that $\beta_0=\beta_1=0$.  Then $\beta_n=\alpha_n$ for all $n$.
\end{lem}

\begin{proof}
First we wish to prove a bound on the size of the coefficients of $G$.  Note that
$$
F(2^k) = \frac{2^{2k}(1+2^k)(1+2^{k-1})\cdots}{(1+2^0)(1+2^{-1})\cdots} = 2^{2k} \prod_{j=1}^k (1+2^k) = O\left(2^{2k+k(k+1)/2} \right).
$$
Now
$$
2^{nk}\beta_n \leq G(2^k) = F(2^k) = O\left(2^{2k+k(k+1)/2} \right).
$$
Therefore
$$
\beta_n = O\left(2^{2k+k(k+1)/2-kn} \right).
$$
Setting $k=n$ we find that
$$
\beta_n = O\left(2^{-n^2/2+5n/2} \right) = O\left( 2^{-\binom{n-2}{2}}\right).
$$
The same can be said for $F$.  Now consider $F-G$.  This is an entire function whose $x^n$ coefficient is bounded by $O\left( 2^{-\binom{n-2}{2}}\right).$  Furthermore $F-G$ vanishes to order at least 2 at 0, and order at least 1 at -1 and at powers of 2.  The bounds on coefficients imply that
$$
|F(x)-G(x)| \leq O\left( \sum_n 2^{-\binom{n-2}{2}} |x|^n \right).
$$
The terms in the above sum clearly decay rapidly for $n$ on either side of $\log_2(|x|)$.  Hence
\begin{align*}
|F(x)-G(x)| & = O\left( 2^{(-\log_2(|x|)^2+5\log_2(|x|))/2+\log_2(|x|)^2}\right) \\
& = O\left( 2^{(\log_2(|x|)^2+5\log_2(|x|))/2}\right).
\end{align*}
In particular $F-G$ is a function of order less than 1.  Hence it must equal
$$
C x^{2+t} \prod_\rho (1-x/\rho),
$$
where the product is over non-zero roots $\rho$ of $F-G$, and $t$ is some non-negative integer.  On the other hand, Jensen's Theorem tells us that if $C\neq 0$ the average value of $\log_2(|F-G|)$ on a circle of radius $R$ is
$$
\log_2 |C| + (2+t)\log_2 R + \sum_{|\rho|<R} \log_2(R/|\rho|).
$$
Setting $R=2^k$ and noting the contributions from $\rho=-1$ and $\rho = 2^j$ for $j<k$ we have
$$
O(1)+3k+\sum_{j<k}(k-j) = O(1) +3k + \binom{k+1}{2} = O(1) + \frac{k^2+7k}{2} > \frac{k^2+5k}{2}
$$
which is larger than $\log_2(|F-G|)$ can be at this radius.  This provides a contradiction.
\end{proof}

We now prove Theorem \ref{mainThmSel}.
\begin{proof}
Suppose that $C(N)$ does not have limit $(\alpha_0,\alpha_1,\ldots)$.  Then there is some subsequence $N_i$ so that $C(N_i)$ avoid some neighborhood of $(\alpha_0,\alpha_1,\ldots)$.  By compactness, $C(N_i)$ must have some subsequence with a limit $(\beta_0,\beta_1,\ldots)$.  By Lemmas \ref{limitLem} and \ref{taylorLem}, $(\alpha_0,\alpha_1,\ldots) = (\beta_0,\beta_1,\ldots)$.  This is a contradiction.

Therefore $\lim_{N\rightarrow\infty} C(N) = (\alpha_0,\alpha_1,\ldots)$.  Hence $\lim_{N\rightarrow\infty}C_d(N) = \alpha_d$ for all $d$.  The Theorem follows immediately from this and the fact the fraction of $b\leq N$ square-free with $(b,D)=1$ that have $|\omega(b)-\log\log N|<(\log\log N)^{3/4}$ approaches 1 as $N\rightarrow \infty$.
\end{proof}

It should be noted that our bounds on the rate of convergence in Theorem \ref{mainThmSel} are non-effective in two places.  One is our treatment in this last Section.  We assume that we do not have an appropriate limit and proceed to find a contradiction.  This is not a serious obstacle and if techniques similar to those of \cite{kn:HB} were used instead, it could be overcome.  The more serious problem comes in our proof of Proposition \ref{characterBoundProp}, where we make use of non-effective bounds on the size of Siegel zeroes. In particular, the rate of convergence depends on the function $Z(\epsilon)$, which is the largest modulus $q$ of a Dirichlet character with a Siegel zero larger than $1-q^\epsilon$ (or 1 if no such $q$ exists). It should then be the case that if for a sufficiently large constant $K$ and integer $m>d$ we have that $N>\exp(Z(K^{-m})^K)$ and $N>e^{e^{e^{Kd}}}$, then
\begin{align*}
\left|\frac{\#\{b\leq N: \dim(S_2(E_b))=d\}}{N} - \alpha_d\right| \leq O_E\left(2^{-\binom{d}{2}}\log\log(N)^{-1/8}+2^{-\binom{d}{2}-m^2} \right).
\end{align*}

\end{document}